\documentclass[a4paper,reqno,12pt]{amsart}
\usepackage[utf8]{inputenc}
\usepackage[english]{babel}
\usepackage{dsfont}
\usepackage{amsmath}
\usepackage{amssymb}
\usepackage{mathtools}
\usepackage{amsthm}
\usepackage{fullpage}
\usepackage{xcolor}
\newtheorem{theorem}{Theorem}[section]
\newtheorem{lemma}[theorem]{Lemma} 
\newtheorem*{lemma*}{Lemma}

\newtheorem{claim}[theorem]{Claim}

\theoremstyle{remark}

\theoremstyle{definition}

\newtheorem{problem}{Problem}

\newcommand{\vass}[1]{\left|#1\right|}
\newcommand{\llb}{\left\lbrace}
\newcommand{\lfl}{\left\lfloor}
\newcommand{\rrb}{\right\rbrace}
\newcommand{\rfl}{\right\rfloor}

\newcommand{\eps}{\varepsilon}
\renewcommand{\epsilon}{\varepsilon}
\renewcommand{\Pr}[1]{\mathbb{P}\left(#1\right)}

\newcommand{\Ex}[1]{\mathbb{E}\left(#1\right)}
\newcommand{\floor}[1]{\left\lfloor#1\right\rfloor}

\newcommand{\se}{\subseteq}

\newcommand{\oldqed}{}
\def\endofClaim{\hfill\scalebox{.6}{$\blacksquare$}}

\title{On product Schur triples in the integers}

\author{Let\'icia Mattos}
\address{Department of Mathematical Sciences, University of Illinois at Urbana-Champaign, Urbana, Illinois
61801, USA (L.~Mattos)}\email{lmattos@illinois.edu}

\author{Domenico Mergoni Cecchelli}
\address{London School of Economics, Department of Mathematics, Houghton Street, London WC2A 2AE, UK (D.~Mergoni Cecchelli)}\email{d.mergoni@lse.ac.uk}

\author{Olaf Parczyk}
\address{Department of Mathematics and Computer Science,  Freie Universität Berlin, Arnimallee~3, 14195 Berlin, Germany (O.~Parczyk)}\email{parczyk@mi.fu-berlin.de }

\thanks{During most of this work, L.~Mattos and O.~Parczyk were supported by the Deutsche Forschungsgemeinschaft (DFG, German Research Foundation) under Germany's Excellence Strategy -- The Berlin Mathematics Research Center MATH+ (EXC-2046/1, project ID: 390685689). D. Mergoni Cecchelli has been partially supported by The London Mathematical Society with Grant Scheme 7, Ref No: SC7-2122-01.}

\usepackage[colorlinks,bookmarks,linkcolor=black,citecolor=black,urlcolor=black]{hyperref}

\begin{document}
	
	\begin{abstract}
	
    Schur's theorem states that in any $k$-colouring of the set of integers $[n]$ there is a monochromatic solution to $a+b=c$, provided $n$ is sufficiently large.
    Abbott and Wang studied the size of the largest subset of $[n]$ such that there is a $k$-colouring avoiding a monochromatic $a+b=c$.
    This led to the exploration of related problems, such as the minimum number of monochromatic $a+b=c$ in $k$-colourings of $[n]$ and the probability threshold for a random subset of $[n]$ to have a monochromatic $a+b=c$ in any $k$-colouring.
	In this paper, we study natural generalisations of these problems to products $ab=c$, in deterministic, random, and randomly-perturbed environments.

	\end{abstract}
\maketitle
	\section{Introduction}
	
	We say that an ordered triple of positive integers $(a,b,c)$, not necessarily distinct, forms a \emph{Schur triple} if $a+b=c$ \footnote{Sometimes in the literature, e.g. Schoen \cite{schoen1999number}, a Schur triple is defined as a set $\{a,b,c\}$ such that $a+b=c$, which does not influence existence statements, but changes the counting, as by our definition the triples $(2,3,5)$ and $(3,2,5)$ are distinct while being represented by the same set $\{2,3,5\}$.}.
	The \emph{Schur number} $S(k)$ is defined as the minimum $n \in \mathbb{N}$ such that for every $k$-colouring of the integers $[n]:=\{1,\ldots,n\}$ there exists a monochromatic Schur triple\footnote{We warn the reader that some authors, e.g. Heule \cite{heule2018schur}, define $S(k)$ as the maximum $n$ such that there exists a $k$-colouring of $[n]$ with no monochromatic Schur triple.}.
	In 1917, Schur~\cite{Schur1917} proved that
	\[\dfrac{3^k+1}{2} \le S(k) \le \lfloor k! e \rfloor.\]
	While the upper bound was later improved to $\lfloor k! (e-\frac{1}{24}) \rfloor$ by Irving~\cite{Irving1973}, determining the exact asymptotics of $S(k)$ has proven to be a difficult problem.
	
	In 1977, Abbott and Wang \cite{AbbottWang1977} introduced a natural problem related to Schur numbers.
	For this, we say that a set $A$ is \emph{$k$-Schur} if every $k$-colouring of $A$ contains a monochromatic Schur triple. 
	In~\cite{AbbottWang1977}, Abbott and Wang posed the question of determining the size of the largest set $A\se [n]$ which is not $k$-Schur, denoted by $g(k,n)$. They showed that 
	\[g(k,n) \geq n-\left \lfloor \dfrac{n}{H(k)} \right \rfloor,\]
    where $H(k)$ is the smallest integer such that every $k$-colouring of $[H(k)]$ contains a monochromatic Schur triple modulo $H(k)+1$.
    Moreover, they conjecture that this is tight and that $S(k) = H(k)$, which together gives the conjecture that $g(k,n) = n - \lfloor n/S(k) \rfloor$.
    While this area has received considerable interest, the conjecture remains in general open.
	For one colour, by considering the size of the set $\{\max(A)-a: a\in A\}$, which is disjoint from $A$, we can conclude that $g(1,n) = \lceil n/2 \rceil$, which proves the conjecture for $k=1$.
	For two colours, the following construction shows that $g(2,n) \ge \lceil 4n/5 \rceil$.
	Let $A$ be the set of all numbers in $[n]$ which are not divisible by 5. In $A$, colour the numbers congruent to 1 or 4 mod 5 red and the remainder blue. One can check that $A$ is a set of size $\lceil 4n/5 \rceil$ and that there is no monochromatic Schur triple in $A$.
	In 1980, Hu \cite{Hu1980} showed that this is the best possible, and hence established the Abbott--Wang Conjecture for two colours. 
	For three or more colours, while similar constructions with modular arithmetic are easy to build and are believed to be optimal, the problem remains open.
	
	In this paper, we introduce product Schur triples and we study what we feel is a natural variation of the Abbott--Wang Conjecture and of Schur's theorem in both the deterministic and the random settings.
    Our work is partially motivated by a question of Prendiville \cite{Prendiville} (see Problem \ref{prob:Prendiville} in the concluding remarks section).
    
    We say that an ordered triple $(a,b,c)$ of positive integers, not necessarily distinct, forms a \emph{product Schur triple} if $ab=c$. A set of numbers is said to be \emph{product-free} if there are no product Schur triples consisting uniquely of elements of that set. 
    For product Schur triples, we can ask similar questions to the ones that have been studied for Schur triples. To start, note that if $n$ is large enough, then every $k$-colouring of $[2,n]$ contains a monochromatic product Schur triple; to see this, we can use the fact that $2^{a}\cdot 2^{b}=2^{a+b}$, paired with Schur's theorem applied to the set of powers of $2$ in $[2,n]=\{2, 3,\dots{},n\}$\footnote{We might abuse notation and denote the set $[2,n]$ with $[n]$.}.
    
    Here and in the following we only focus on product-free subsets of $[2,n]$, as $(1,1,1)$ and triples of the form $(1,a,a)$ and $(a,1,a)$ are trivial product Schur triples that we do not wish to include in our counting results\footnote{Similarly to how $0$ is excluded in the literature when counting Schur triples.}.
 
    \subsection{The deterministic setting}
    
    For a given positive integer $k$, and $n$ large enough, we define $g_*(k,n)$ to be the smallest size $s$ such that for any given $A\subseteq [2,n]$ of size at least $s$ and for every $k$-colouring of $A$, the colouring contains a monochromatic product Schur triple.
    Our first result provides upper and lower bounds on $g_*(k,n)$.
    Our bounds depend on Schur numbers and on what we call the \emph{double-sum Schur number}, which was previously investigated by Abbott and Hanson \cite{abbott1972problem} under their analysis of what they called \emph{strongly sum-free sets}.
    We call the \emph{double-sum Schur number} $S'(k)$ the minimum $n \in \mathbb{N}$ such that for every $k$-colouring of $[n]$ there exists a monochromatic solution\footnote{Here and in the following, given a coloured set $X$ and an equation (e.g. $a+b=c$), we call a monochromatic solution of $a+b=c$ in $X$ an ordered tuple $(x_1, x_2, x_3)$ of elements of $X$ with the same colour such that $x_1+x_2=x_3$ (i.e. that the elements respect the equation in their order). E.g. in a $1$-colouring of $[5]$, both $(2,3,5)$ and $(3,2,5)$ are solutions to $a+b=c$, while $(3,5,2)$ is not.} of $a+b=c$ or $a+b=c-1$.

	\begin{theorem}\label{thm:det_prod}\label{prop:det_prod}
    Let $\eps >0$, and let $k$ be a positive integer.
    For every $n>(\frac{2}{\eps})^{S(k)^2}$ we have 
		\[n-n^{1/S'(k)}\le g_*(k,n) \le n-(1-\eps) n^{1/S(k)}.\]
	\end{theorem}
	
	The proof of Theorem~\ref{thm:det_prod} is given in Section~\ref{sec:det}. As $S'(k)=S(k)$ for $k \in \{1,2,3\}$ (we have $S(1)=2,\ S(2)=5,\ S(3)=14$), Theorem~\ref{thm:det_prod} is optimal for these values. For $k=4$, we have $S'(4)=41<45=S(4)$, while for $k>4$ precise values of $S'(k)$ are not known and values of $S(k)$ have proved extremely difficult to calculate precisely \cite{heule2018schur}.
	
	If, instead of seeking the minimum size that forces a subset of $[n]$ to contain a monochromatic Schur triple in any $k$-colouring, we consider $k$-colourings in $[n]$ which minimise the number of monochromatic Schur triples, we recover a problem that was first investigated in 1996 by Graham, R\"odl and Ruci\'nski~\cite{GrahamRodlRucinski1996}. They used Goodman's result to show that there are at least $n^2/19+O(n)$ monochromatic Schur triples in any 2-colouring of $[n]$\footnote{Note that here and in the following paragraph, by the definition of Schur triple that we follow, the result we report are shifted by a factor $2$, as the authors considered Schur triples as sets and not ordered triples.}. 
	In the late 1990s, Schoen~\cite{schoen1999number} and, independently, Robertson and Zeilberger~\cite{Robertson1998} improved this bound to $n^2/11+O(n)$. This is best possible, as the colouring $[n]=R\cup B$ given by $R=(\frac{4n}{11},\frac{10n}{11}]$ and $B = [n]\setminus R$ contains $n^2/11+O(n)$ monochromatic solutions for $a+b=c$.
	Similar results for larger values of $k$ are not known and would be of great interest.

    In the same vein as the results of Schoen~\cite{schoen1999number} and Robertson and Zeilberger~\cite{Robertson1998}, Prendiville \cite{Prendiville} asked for the minimum number of monochromatic product Schur solutions in every 2-colouring of $[2,n]$.
    Our next theorem shows that it is at least $n^{1/3-o(1)}$.
    
	\begin{theorem}\label{thm:det_prod_mult}
		For every $\varepsilon>0$ there exists $n_0(\eps)\in \mathbb{N}$ such that the following holds for all $n \ge n_0(\eps)$. Every $2$-colouring of $[2,n]$ contains at least $n^{1/3-\varepsilon}$ monochromatic product Schur triples.
	\end{theorem}

    Whilst writing the paper, Aragão, Chapman, Ortega and Souza~\cite{ACOS23} showed, among other results, that any $2$-colouring of $[2,n]$ has at least $(2^{-3/2}-o(1))n^{1/2}\log(n)$ monochromatic integer solutions to ${ab=c}$, which is sharp up to lower order terms.
	In general, we believe that the number of monochromatic product Schur triples in any $k$-colouring of $[2,n]$ should depend on Schur numbers.
    We observe that the $k$-colouring on $[2,n]$ that provides the lower bound on $g_*(k-1,n)$ in Theorem~\ref{thm:det_prod} yields $O(n^{1/S'(k-1)} \log(n))$ monochromatic product Schur triples, which matches the lower bound of Aragão, Chapman, Ortega and Souza~\cite{ACOS23} up to $\log$ factors also for $k=3,4$.
	Theorems~\ref{thm:det_prod} and~\ref{thm:det_prod_mult} are proven in Section~\ref{sec:det}.
	
	\subsection{Product Schur triples in random sets} 
	
	Amongst the subsets of $[n]$ which do not contain a Schur triple, there are exactly two which attain the maximal size of $\lceil n/2 \rceil$. These are the set of the odd numbers in $[n]$ and the set $(n/2,n]\cap \mathbb{N}$. This was already known to Cameron and Erd\H{o}s~\cite{CameronErdos}. Therefore, a `typical' set of size $n/2$ contains a Schur triple.
	From a probabilistic point of view, it is natural to go one step further and investigate for which densities a typical random set contains a Schur triple.
	In order to formalise what we mean by this, we first need some preparation.
	
	For a set $A \subseteq \mathbb N$, let $A_p$ be the random set that is obtained by including each element of $A$ independently at random, each with probability $p$.
	For a collection $\mathcal{P}$ of subsets of $\mathbb{N}$, usually referred to as a property, we say that a function $\hat{p}: \mathbb{N} \to [0,1]$ is a \emph{threshold} for $\mathcal{P}$ in $A$ if
	\[\Pr{A_p \in \mathcal{P}} \to \begin{cases} 0, \text{ if } p \ll \hat{p};\\ 1, \text{ if } p \gg \hat{p}.\end{cases}\]
	Above, the notation $p \ll \hat{p}$ stands for $p = o(\hat{p})$.
	A well-celebrated result of Bollob\'as and Thomason~\cite{bollobas1987threshold} states that thresholds exist for non-trivial monotone properties. Moreover,
	if $\hat{p}_\alpha$ and  $\hat{q}_\alpha$ are both thresholds for a property $\mathcal{P}$, then $\hat{p}_\alpha = \Theta (\hat{q}_\alpha)$.
	For this reason, we abuse notation and refer to a threshold function as \emph{the} threshold function, when $A$ and $\mathcal{P}$ are clear from the context.
	
	A routine application of the first and second moment methods shows that the threshold in $[n]_p$ to contain a Schur triple is $n^{-2/3}$. It is natural to expand this result and ask about the threshold for $[n]_p$ to contain solutions for certain equations. 

    Since the set of product Schur triples in $[2,n]$ is extremely sparse, we do not expect a direct application of the hypergraph container method or the second-moment method to work. 
	Nevertheless, we are able to overcome this problem and find the threshold for $[2,n]_p$ to contain a solution for $ab=c$.
    This is the first result on solutions of nonlinear equations in random sets that we are aware of.

	\begin{theorem}\label{thm:threshold}
		The threshold for $[2,n]_p$ to contain a product Schur triple is $(n\log(n))^{-\frac{1}{3}}$.
	\end{theorem}

 The lower bound, i.e. the statement that for $p\ll (n\log(n))^{-\frac{1}{3}}$, with high probability the set $[2,n]_p$ does not contain a product Schur triple, is simply an application of the first-moment method. The challenge here is to show the corresponding upper bound. The proof of Theorem~\ref{thm:threshold} is given in Section~\ref{sec:random}.

    \subsection{Product Schur triples in randomly perturbed sets} 
	
	One may also ask how much we need to randomly perturb a set in order to obtain a given property.
	The origin of this type of question dates back to the work of Bohman, Frieze and Martin~\cite{bohman2003many}. They investigated the number of random edges that must be added to an arbitrary dense graph in order to make the resulting graph Hamiltonian with high probability\footnote{An event that occurs with high probability is one whose probability depends on a certain number $n$ and tends to one as $n$ tends to infinity.}.

	In 2018, Aigner-Horev and Person~\cite{aigner2019monochromatic} initiated the study of randomly perturbed structures in additive combinatorics, specifically in the context of 2-colouring the set $[n]$.
	They showed that if $A \se [n]$ is a dense set and $p \gg n^{-2/3}$, then with high probability any 2-colouring\footnote{It is also natural to consider the corresponding `non-coloured' Schur problem, but in this case the problem is much simpler. If $A$ is dense and $p \gg n^{-1}$, then $A \cup [n]_p$ contains a Schur triple with high probability, since $(A-A)\cap [n]$ is also dense. This is best possible, as $[n]_p = \emptyset$ with high probability if $p \ll n^{-1}$ and the set of odd numbers does not contain a Schur triple.} of $A \cup [n]_p$ contains a monochromatic Schur triple.
    This is best possible since $[n]_p$ does not contain a Schur triple with high probability if $p \ll n^{-2/3}$, and therefore we could for example take $A$ the set of odd numbers, colour it with red, and colour $[n]_p \setminus A$ with blue.
	In a recent development, Das, Knierim, and Morris~\cite{das2023schur} refined the work in~\cite{aigner2019monochromatic} and analysed random perturbations of sets with sizes ranging between $\sqrt{n}$ and $\eps n$.
	
	Inspired by the work in~\cite{aigner2019monochromatic} and~\cite{das2023schur}, we initiate the study of product Schur triples in randomly perturbed sets.
	For a function $\alpha: \mathbb{N} \to (0, 1]$, let $\hat{p}_\alpha:  \mathbb{N} \to (0, 1)$ be a function such that the following holds:
	\begin{itemize}
		\item[$(A)$] There exists a sequence of sets $(C_n)_{n \in \mathbb{N}}$ with $C_n \se [2,n]$ such that $|C_n| \ge (1-\alpha(n)) n$ and such that for all $p \ll \hat{p}_{\alpha}$ we have
		\[\lim_{n\to \infty}\Pr{C_n \cup [2,n]_p \text{ contains a product Schur triple}}=0.\]
		\item[$(B)$] For all sequences of sets $(C_n)_{n \in \mathbb{N}}$ with $C_n \se [2,n]$ and $|C_n| \ge (1-\alpha(n)) n$ and for all $p \gg \hat{p}_{\alpha}$ we have
		\[\lim_{n\to \infty}\Pr{C_n \cup [2,n]_p \text{ contains a product Schur triple}}=1.\]
	\end{itemize}
	We refer to $\hat{p}_\alpha$ as a threshold for the $\alpha$-randomly perturbed product Schur property.
	Observe that if $\alpha \equiv 1$, then $\hat{p}_{\alpha}$ can be taken to be a threshold for $[2,n]_{p}$ to contain a product Schur triple. If $\alpha  \le 1/\sqrt{n}$, then $C_n$ has a product Schur triple and hence $\hat{p}_{\alpha}$ can be taken to be $\hat{p}_{\alpha}=0$ (note that $(A)$ is a vacuous statement in this case).
	In general, $\hat{p}_{\alpha}$ is known to exist for every non-increasing function $\alpha$, see~\cite{bollobas1987threshold}.
	Similar to the standard threshold function, if  $\hat{p}_\alpha$ and  $\hat{q}_\alpha$ are functions that satisfy both items $(A)$ and $(B)$, then $\hat{p}_\alpha = \Theta (\hat{q}_\alpha)$.
	For this reason, again we slightly abuse notation and refer to a function $\hat{p}_{\alpha}$ satisfying $(A)$ and $(B)$ as \emph{the} threshold function for the $\alpha$-randomly perturbed product Schur property.
    
    To state our first result in the randomly perturbed model, we need the following constant related to the number of integers in $[n]$ that have a divisor within a given interval. The order of magnitude of this quantity was established by Ford~\cite{Ford2008}:
	\begin{align}\label{eq:delta}
		\delta = 1 - \dfrac{1+\log \log (2)}{\log(2)} \sim 0.086071 \, .
	\end{align}
    The smallest function $\alpha$ in which we were able to find the threshold $\hat{p}_{\alpha}$ for having a product Schur triple in randomly perturbed sets is of order $\alpha = (\log(n))^{-\delta+o(1)}$.
    
    \begin{theorem}\label{thm:corollary}
        For $(\log(n))^{-\delta} (\log \log(n))^{3/2+\delta} \le \alpha=o(1)$ we have $\hat{p}_\alpha(n) = n^{-1/2+o(1)}$.
    \end{theorem}

    We have also obtained bounds for the threshold for the $\alpha$-randomly perturbed product Schur property for a wide range of $\alpha$, which match when $\alpha$ is constant.
    Moreover, it interpolates from the roughly $n^{-1/2}$ from Theorem~\ref{thm:corollary} in the direction of $n^{-1/3}$, which is roughly the threshold in $[2,n]_p$ alone (recall Theorem~\ref{thm:threshold}).
    
    Our result depends on two functions $f: (0,1) \to \mathbb{R}$ and $\beta: (0,1)\to \mathbb{R}$, where $f(\alpha)$ is chosen such that
	\begin{align}\label{eq:defn-f}
	    (4f(\alpha))^{\delta} \log (1/(2f(\alpha)))^{-3/2}  = \alpha \qquad \text{and} \qquad \beta(\alpha) = \dfrac{f(\alpha)}{1+2f(\alpha)} \, .
	\end{align}
    We note that technical, but elementary calculations for $\alpha \in (0,2^{-7})$ give that $f(\alpha)\le 1/4$ and
    \begin{align}
    \label{eq:f_alpha}
    \alpha^{1/\delta} \left( \log\left(2 \alpha^{-1/\delta}\right) \right)^{-3/(2\delta)} \le f(\alpha) \le \alpha^{1/\delta} \left( \log\left(2 \alpha^{-1/\delta}\right) \right)^{3/(2\delta)} \, .
    \end{align}
    For a proof of this, see Claim~\ref{eq:growth-of-f}.
    
    Our general theorem in the randomly perturbed setting can be stated as follows. Observe that $\alpha$ varies from a logarithmic-like range to a constant range.

	\begin{theorem}\label{thm:perturbed}
	There exists a constant $0<\gamma \le 1$ such that for any $\alpha$ with $$(\log(n))^{-\delta} (\log \log(n))^{3/2+\delta}\leq \alpha \leq2^{-7},$$
    the following holds:
		\begin{enumerate}
			\item [$(i)$] There exists a sequence of sets $(C_n)_{n \in \mathbb{N}}$ with $|C_n| \ge (1-2\gamma^{-1} \alpha)n$ such that for all $p \ll  n^{-\frac{1}{2}+\beta(\alpha)}$ we have
			\[\lim_{n\to \infty}\Pr{C_n \cup [2,n]_p \text{ contains a product Schur triple}}=0.\]
			\item [$(ii)$] For any sequences of sets $(C_n)_{n \in \mathbb{N}}$ with $|C_n| \ge \left (1- 2^{-1}\gamma\alpha\right )n$ and for all $p \gg  \alpha^{-1} n^{-\frac{1}{2}+\beta(\alpha)}$ we have
			\[\lim_{n\to \infty}\Pr{C_n \cup [2,n]_p \text{ contains a product Schur triple}}=1.\]
		\end{enumerate}
	\end{theorem}

    If $(\log(n))^{-\delta} (\log \log(n))^{3/2+\delta} \le \alpha = o(1)$, then Theorem~\ref{thm:perturbed} directly implies Theorem~\ref{thm:corollary}.
    On the other side, when $\alpha\ge \tfrac 12 \gamma$, $C_n$ might be empty and in that case the threshold is $(n\log(n))^{-1/3}$ as follows from Theorem~\ref{thm:threshold}.
    We believe that our theorem actually holds true with $\gamma=1$ and can further be improved such that the exponent of $n$ in $\hat{p}_{\alpha}$ tends towards $-\tfrac13$ as $\alpha$ increases.
    We note that numerically we can find more accurate estimates of $f(\alpha)$ for $\alpha\ge 2^{-7}$ compared to \eqref{eq:f_alpha}, but then still the $\gamma$ is the bottleneck.
	
	\medskip
	
	\textbf{The rest of this paper is divided as follows.}
	In Section~\ref{sec:det} we prove Theorems~\ref{thm:det_prod} and~\ref{thm:det_prod_mult}; in Section~\ref{sec:random} we prove Theorem~\ref{thm:threshold}; in Section~\ref{sec:perturbed} we prove Theorem~\ref{thm:perturbed}; and in Section~\ref{sec:remarks} we state some open problems.


	\section{Product Schur in deterministic sets}
	\label{sec:det}
	
    We begin this section with the proof of Theorem~\ref{thm:det_prod}.
 
	\begin{proof}[Proof of Theorem~\ref{thm:det_prod}.]
        We first prove the upper bound.
        Let $\varepsilon>0$ and $k,n \in \mathbb{N}$.
		Assume for a contradiction that we can find a set $A\subseteq [2,n]$ of size at least $n-(1-\eps) n^{1/S(k)}$ which can be partitioned into $k$ product-free sets $A_1,\dots{},A_k$. Fix $A'=[\tfrac12 \eps n^{1/S(k)}, n^{1/S(k)}]$ which has size $(1-\tfrac12 \eps)n^{1/S(k)}$.
        
        Importantly, for distinct $a,b\in A'$ and any choice of $i,j\in [S(k)]$ we have that $a^i\neq b^j$.
		Indeed, without loss of generality we have $j>i$ and it suffices to show that for all $b\in A'$ we have $b^\frac ji>n^{1/S(k)}$, because this implies $b^\frac ji\not\in A'$ and therefore $a^i\neq b^j$ for all $a\in A'$.
        
		Because $\tfrac12 \eps n^{1/S(k)}\in A'$ and for any $b\in A'$ we have $b^{\frac{j}{i}}\geq (\tfrac12 \eps n^{1/S(k)})^{\frac ji}$, our statement is equivalent to showing that $(\tfrac12 \eps n^{1/S(k)})^{\frac ji}>n^{1/S(k)}$ for any choice of $j>i$ in $[S(k)]$.
        It suffices to verify the case $j=S(k),\ i=S(k)-1$ and a short calculations shows that this holds because $(\tfrac12 \eps)^{S(k)^2}>n^{-1}$ by assumption.
		
		Next, we show that there is an element $a$ in $A'$ such that $P(a):=\llb a^i : i=1,\dots{}, S(k)\rrb$ is contained in $A$.
		Indeed, notice that for all $a$ in $A'$ we have $P(a)\subseteq [2,n]$.
        Moreover, if $a, a'\in A'$ are distinct, then $P(a)$ and $P(a')$ are disjoint.
        Therefore, if for each one of the elements of $A'$ a different element of $[2,n]$ was missing from $A$, we would get
		\[|A| \le n - |A'| = n-(1-\tfrac12 \eps) n^{1/S(k)} \, . \]
		
		Now fix an $a\in A'$ such that $P(a)\subseteq A$.
		By applying $\log_a(\cdot)$ to the elements of $P(a)$, the partition $A_1, \dots, A_k$ of $A$ restricted to $P(a)$ induces a partition $S_1, \dots, S_k$ of $[S(k)]$.
		As the partition $A_1,\dots,A_k$ is product-free, the partition $S_1, \dots, S_k$ is sum-free, in contradiction to the definition of $S(k)$.

        It remains to give the construction of a colouring for the lower bound.
		For an integer $k$ let $\chi : [S'(k)-1] \rightarrow [k]$ be a $k$-colouring of $[S'(k)-1]$ without monochromatic $a+b=c$ and $a+b=c-1$.
		We colour each integer $a \in (n^{1/S'(k)},n]$ with colour $\chi(\lceil S'(k) \cdot \log_n(a) \rceil - 1)$.
		For a contradiction assume that there is a monochromatic product $ab=c$ in this colouring.
		Then let $a'=\lceil S'(k) \cdot \log_n(a) \rceil -1$,  $b'=\lceil S'(k) \cdot \log_n(b) \rceil -1$, and $c'=\lceil S'(k) \cdot \log_n(c) \rceil -1$ and note that $\log_n(a)+\log_n(b)=\log_n(c)$ implies $a'+b'=c'$ or $a'+b' =c'-1$.
		But as $ab=c$ was monochromatic we have $\chi(a')=\chi(b')=\chi(c')$, a contradiction.
	\end{proof}

    In order to prove Theorem~\ref{thm:det_prod_mult}, we need the following supersaturation lemma.
    This lemma is sharp up to a constant factor, as the set $[n] \setminus [ \lfloor \frac{1}{2}\sqrt{n} \rfloor]$ contains at most $4n$ product Schur triples.
    Indeed, if $a,b,c \in [n] \setminus [ \lfloor \frac{1}{2}\sqrt{n} \rfloor]$ are such that $ab = c $, then $\sqrt{n}/2 \le a,b$ and hence $a,b \le 2 \sqrt{n}$, which implies that the number of product Schur triples in the set is at most $(2\sqrt{n})^2 = 4n$.

	\begin{lemma}\label{thm:supersat}
		For every $\varepsilon>0$ there exists $n_0(\eps)\in \mathbb{N}$ such that the following holds for all $n \ge n_0(\eps)$. If $A\subseteq [2,n]$ is a set of size at least $n-\tfrac 12 \sqrt{n}$, then there are at least $n/8$ solutions in $A$ to $ab=c$.
	\end{lemma}

Another tool needed in the proof of Theorem~\ref{thm:det_prod_mult} and also to prove Lemma~\ref{thm:supersat} is the following consequence of a result of Wigert~\cite{wigert1907ordre} on the number of divisors of an integer.
	
	\begin{lemma}
		\label{lem:divisors}
		Every integer $n$ has at most $e^{O\left(\frac{\log(n)}{\log \log(n)}\right)}$ divisors.
	    In particular, for every $\varepsilon>0$ there exists $n_0(\eps) > 0$ such that if $n \ge n_0(\eps)$, then $n$ has at most $n^{\eps}$ divisors.
	\end{lemma}

	\begin{proof}[Proof of Lemma~\ref{thm:supersat}]
    Let us write $B=A\cap [\sqrt{n}]$ and $C = [n]\setminus A$ and note that $${|B| \ge \frac{\sqrt{n}}{2}\ge |C|}.$$
    Let $\mathcal{A}$ be the set of triples $(a,b,c)\in B\times B\times A$ such that $ab=c$ and let $\mathcal{C}$ be the set of triples $(a,b,c) \in B\times B \times C$ such that $ab=c$.
    Our main goal is to lower bound the size of $\mathcal{A}$. For this, we first note that 
    $$|\mathcal{A}|+|\mathcal{C}| = |\{(a,b,c)\in B\times B\times [n]: ab=c\}| = |B|^2 \ge n/4.$$
    In the last inequality, we used that $|B|\ge \sqrt{n}/2$.
    This implies that $|\mathcal{A}|\ge n/4-|\mathcal{C}|$, and hence it suffices to upper bound the size of $\mathcal{C}$.
    Now, note that $|\mathcal{C}|$ is at most the number of solutions of $ab=c$ with $c \in C$.
    By Lemma~\ref{lem:divisors}, as each $c \in C$ contains at most $n^\eps$ divisors, it follows that $|\mathcal{C}| \le n^{\eps}|C| \le n^{1/2+\eps}$.
    We conclude that the size of $\mathcal{A}$ is at least $n/4-n^{1/2+\eps}\ge n/8$, as required.\end{proof}

 

	
	\begin{proof}[Proof of Theorem~\ref{thm:det_prod_mult}]

    Let us fix $\epsilon \in \left(0,\frac{1}{12}\right)$, take $n$ to be large enough, and fix a red-blue colouring of $[2,n]$. We denote by $R$ the set of numbers in $[\lfloor n^{1/3} \rfloor]$ that are coloured red, and by $B$ the set of those coloured blue.
    Without loss of generality, we assume that $|R| \ge |B|$.

    If $|B|< n^{1/6}/2$, then $|R| \ge n^{1/3} - n^{1/6}/2$, and hence by Lemma~\ref{thm:supersat} we would have at least $n^{1/3}/8$ red product Schur triples.
    Thus, we may assume from now on that $|R| \ge |B| \ge n^{1/6}/2$.
    Set 
    \begin{align*}
        P_{R}=\llb ab : a, b\in R\rrb \qquad \text{and} \qquad P_{B}=\llb ab : a, b\in B\rrb.
    \end{align*}
    By Lemma~\ref{lem:divisors} we have that these two sets have size at least $n^{1/3-\varepsilon}$.
    Moreover, we may assume that $P_{R}$ contains at least $n^{1/3-\varepsilon}/2$ blue elements and that $P_{B}$ contains at least $n^{1/3-\varepsilon}/2$ red elements, otherwise we are done.

    For a set $\{s_1,s_2,s_3,s_4\}$, we say that $(a,b,c)$ is a product Schur triple \emph{associated} to it if there exist distinct indices $i,j,k \in \{1,2,3,4\}$ such that $a=s_i$, $b=s_js_k$ and $c=s_is_js_k$.
    Now, we define $\mathcal{S}$ to be the set of all pairs $\big ((a,b,c),\{r_1,r_2,b_1,b_2\} \big )$ with the following properties:
        
        \vspace*{2pt}
        
        (i) $r_1, r_2 \in R$ and $b_1,b_2 \in B$ (all distinct), $r_1r_2 \in B$ and $b_1b_2 \in R$;
        
        (ii) $(a,b,c)$ is a product Schur triple associated to $\{r_1,r_2,b_1,b_2\}$.

        \vspace*{2pt}

        We claim that if $\{r_1,r_2,b_1,b_2\}$ is a set as in (i), then
        there exists a monochromatic product Schur triple associated to $\{r_1,r_2,b_1,b_2\}$.
        In fact, if $r_1r_2b_1$ is blue, then $(b_1,r_1r_2, r_1r_2b_1)$ is a blue product Schur triple, and if $b_1b_2r_1$ is red, then $(r_1,b_1b_2, b_1b_2r_1)$ is a red product Schur triple.
        Thus, we may assume that this is not the case, and hence we have that $r_1r_2b_1$ is red and that $b_1b_2r_1$ is blue.
        Now, if $r_1b_1$ is red, then $(r_2,r_1b_1,r_1r_2b_1)$ is a red product Schur triple; if $r_1b_1$ is blue, then $(b_2,r_1b_1, b_1b_2r_1)$ is a blue product Schur triple.
        This proves our claim.

        As we are assuming that $P_R = \{ ab: a,b \in R \}$ contains at least $n^{1/3-\eps}/2$ blue elements and that $P_B = \{ ab: a,b \in B \} $ contains at least $n^{1/3-\eps}/2$ red elements, we have at least $n^{2/3-2\eps}/4$ sets $\{r_1,r_2,b_1,b_2\}$ as in (i).
        This together with our previous claim implies that 
        \begin{align}\label{eq:lower-bound-S}
            |\mathcal{S}| \ge n^{2/3-2\eps}/4.
        \end{align}

        Now fix a monochromatic product Schur triple $(a,b,c)$. By Lemma~\ref{lem:divisors} there are at most $n^{\eps}$ ways to write $c$ as a multiplication of three numbers, say $c=s_1s_2s_3$.
        Given $s_1,s_2$ and $s_3$, there are at most $\max\{|R|,|B|\} \le |R|$ ways to choose a fouth element $s_4$ so that $\big ((a,b,c),\{s_1,s_2,s_3,s_4\} \big)$ is in $\mathcal{S}$.
        Thus,
        \begin{align}\label{eq:upper-bound-S}
            | \mathcal{S} | \le n^{\eps} |R| \# \{ \text{monochromatic product Schur triples}\}.
        \end{align}
        By combining~\eqref{eq:lower-bound-S} and~\eqref{eq:upper-bound-S}, and using that $|R| \le n^{1/3}$, 
        we obtain that the number of monochromatic product Schur triples is at least 
        \begin{align*}
            \dfrac{n^{2/3-2\eps}}{4n^{\eps}|R|} \ge n^{1/3-4\eps}.
        \end{align*}
        This concludes our proof.
	\end{proof}

	\section{Product Schur triples in random sets}
	\label{sec:random}
	
	In this section, we prove Theorem~\ref{thm:threshold}.
	

	\begin{proof}[Proof of Theorem~\ref{thm:threshold}]
        To lower bound the threshold we want to estimate the expected number of product Schur triples.
        In $[2,n]$ there are at most $\sqrt{n}$ product Schur triples $(a,b,c)$ with $a=b$.
        We denote the remaining triples by $T_n$ and note that $|T_n|$ is exactly the number of ordered pairs $(a, b)$ of elements of $[2,n]$ such that $a\cdot b\leq n$ and $a\neq b$. 
        Next, we count $|T_n|/2$, which is precisely the number of triples $(a,b,c)$ with $a<b$ and $a\cdot b\leq n$. Note that
        
		\begin{align*}
			\frac{1}{2}\vass{T_n} &=\sum_{a=2}^{\floor{\sqrt{n}}}\vass{\llb b\in [2,n]: a< b,\ a\cdot b\leq n\rrb}=\sum_{a=2}^{\floor{\sqrt{n}}}\vass{\llb b\in [2,n]: a\cdot b\leq n\rrb\setminus [a]}.
		\end{align*}
	As $\vass{\llb b\in [2,n]: a\cdot b\leq n\rrb}=\floor{\frac{n}{a}}$, it follows that
		\begin{align*}
			\frac12\vass{T_n}&=\sum_{a=2}^{\lfl\sqrt{n}\rfl} \left( \floor{\frac{n}{a}}-a\right)=\sum_{a=2}^{\lfl\sqrt{n}\rfl} \floor{\frac{n}{a}} + O(n) =\sum_{a=2}^{\lfl\sqrt{n}\rfl} \frac{n}{a}+ O(n).
		\end{align*}
		As the harmonic numbers $H_x=\sum_{i=1}^x \frac1i$ asymptotically behave like $\log (x)$, it follows that $\vass{T_n} = (1 + o(1)) n\log(n)$.
 
        Now let $X$ be the random variable which counts the number of product Schur triples in $[2,n]_p$.
        We have
		\begin{align*}
			\Ex{X}=\sum_{ (a,b,c)\in T_n}\Pr{(a,b,c)\in  [2,n]_p} + O(p^2 \sqrt{n})=O(p^3n\log(n)+p^2 \sqrt{n}).
		\end{align*}
		Thus, if $p\ll (n\log(n))^{-1/3}$, then $\Ex{X}\ll 1$.
		By Markov's inequality\footnote{Markov's inequality states that if $X$ is a non-negative random variable and $t >0$, then $\Pr{X\ge t}\le \Ex{X}/t$.}, it follows that $\Pr{X\ge 1} \to 0$ if $p\ll (n\log(n))^{-1/3}$. This implies that $\hat{p}(n) \ge (n\log(n))^{-1/3}$.
		
		In order to prove that $\hat{p}(n)\le (n\log(n))^{-1/3}$, we consider two independent copies of a random set.
        As containing a product Schur triple is a monotone property, we can assume that $p = f(n)(n\log(n))^{-1/3}$, where $f(n)\to \infty$ as $n \to \infty$ but $f(n)\le \log(n)$.
		Let $q \in [0,1]$ be such that $(1-q)^2 = 1-p$; note that $q$ is asymptotically equal to $p/2$.
		Let $A := [2,n]_{q}$ and $B := [2,n]_{q}$ to be two independent random sets and set $C = A \cup B$. 
		Observe that $q$ was chosen so that $C$ has the same distribution as $[2,n]_p$.

		To show that $C$ contains a product Schur triple with high probability, we claim that it suffices to show that $\big| A^2\cap [2,n] \big |\gg 1/q$ with high probability.
		Indeed, set $Y = \big| A^2 \cap B \cap [2,n] \big |$.
		Observe that $Y \ge 1$ if and only if there exist $x,y \in A$ (not necessarily distinct) and $z \in B$ such that $xy =z$.
		If $\big| A^2\cap [2,n] \big |\gg 1/q$, then $\mathbb{E}_B(Y)\gg 1$, and hence by Chernoff's inequality\footnote{Chernoff's inequality states that if $X$ is a binomial random variable and $t \ge 0$, then $\Pr{|X-\Ex{X}|\ge t}\le 2e^{-t^2/(2\Ex{X}+t)}.$} we have
		\[\mathbb{P}_B\big(Y = 0\big) \le e^{-\omega(1)}.\]
		Thus,  $Y\ge 1$ with high probability as long as $\big| A^2\cap [2,n] \big |\gg 1/q$ with high probability.

		Next, we show that for a typical set $A$, no number $c \in [2,n]$ should have more than two representatives $(a,b) \in A \times A$ such that $a \le b$ and $ab = c$.
		Indeed, for each $c \in [2,n]$ consider the set of representatives of $c$ given by
		\[P_c = \{(a,b): a, b \in [2,n], a \le b \text{ and } ab =c\}.\]
		Let $\eps \in (0,10^{-1})$ be a constant. 
		If $n$ is sufficiently large, then number of divisors of $c$ is at most $O(n^{\eps})$, for all $c \in [n]$ by Lemma~\ref{lem:divisors}.
		Thus, $|P_c| = O(n^{\eps})$ and hence $P_c$ has at most $O(n^{3\eps})$ subsets of size three.
		For each $\{(a_i,b_i): i \in [3]\} \se P_c$, the probability that $(a_i,b_i) \in A \times A$ for all $i \in [3]$ is at most $q^5$, as one of the elements can be repeated in case $c$ is a perfect square. Thus, we have
		\begin{align}\label{eq:PcinterAtimesA}
			\Pr{|P_c \cap (A \times A)|\ge 3} = O(n^{3\eps}q^{5})
		\end{align}
		for all $c \in [2,n]$.
		By~\eqref{eq:PcinterAtimesA} combined with the union bound, it follows that the event that there exists a $c \in [2,n]$ for which $|P_c \cap (A \times A)|\ge 3$ has probability at most $O(n^{1+3\eps}q^{5})$. This tends to $0$ as $n$ tends to infinity, and hence $|P_c \cap (A \times A)|\le 2$ for all $c \in [2,n]$ with high probability.
		
		From the discussion above, it follows that
		\[\big|A^2 \cap [2,n] \big| \ge \frac{1}{2} \Big |\Big \{(a,b): a,b \in A, a \le b \text{ and } ab \le n\Big\} \Big |\]
		with high probability.
		In other words, we have
		\[\big|A^2 \cap [2,n] \big| \geq \frac12\sum \limits_{a \in A \cap [\sqrt{n}]} \big | [a,n/a]\cap A \big | = \frac12\sum \limits_{a \in [2,\sqrt{n}]} \big | [a,n/a] \cap A \big | \mathds{1}_{a \in A}\]
		with high probability.
		Now, note that by Chernoff's inequality, we have with high probability that for every $a\le \sqrt{n}/2$ it holds that
		\[\big| [a,n/a] \cap A \big|= (1\pm 2^{-1})q\left(\dfrac{n}{a}-a\right) \ge \dfrac{qn}{4a}.\]
		This implies that 
		\begin{align}\label{eq:boundonA2}
			|A^2\cap [2,n]| \ge \dfrac{qn}{8} \sum \limits_{a \in [2,\sqrt{n}/2]}  \dfrac{\mathds{1}_{a \in A}}{a}
		\end{align}
		with high probability.
		
		Now we bound the right-hand side of~\eqref{eq:boundonA2}.
		We decompose almost the whole interval $[2,\sqrt{n}/2]$ into disjoint sub-intervals of size $1/q$. Note that
		\[ [2,\sqrt{n}/2] \supseteq \bigcup_{i=1}^{\lfloor q\sqrt{n}/4\rfloor} (i/q,(i+1)/q] \, . \]
		Then, we have
		\begin{align}\label{eq:boundsuminv}
			\sum \limits_{a \in [2,\sqrt{n}/2]}  \dfrac{\mathds{1}_{a \in A}}{a} \ge 
			\sum \limits_{i=1}^{\lfloor q\sqrt{n}/4\rfloor} \dfrac{q}{i+1} \mathds{1}_{A \cap (i/q,(i+1)/q] \neq \emptyset}.
		\end{align}
		As the size of the interval $ (i/q,(i+1)/q] \cap \mathbb{N}$ is of order $1/q$ and $A=[2,n]_{q}$, we have
		\begin{align}\label{eq:Ainterval}
			\mathbb{P}\big( A \cap  (i/q,(i+1)/q] = \emptyset \big) \sim (1-q)^{1/q} \sim e^{-1}. \, 
		\end{align}
		
		By simplicity, denote
		\[S : = \sum \limits_{i=2}^{\lfloor q\sqrt{n}/4\rfloor} \dfrac{J_i}{i},\]
		where $J_i := \mathds{1}_{A \cap ((i-1)/q,i/q] \neq \emptyset}$
		for every $i\ge 2$.
		Note that $(J_i)_i$ are independent and identically distributed Bernoulli random variables with constant probability, see~\eqref{eq:Ainterval}.
		By combining~\eqref{eq:boundonA2} and~\eqref{eq:boundsuminv}, it follows that 
		\begin{align}\label{eq:A2intn-final}
			|A^2\cap [2,n]|  \ge q^2nS/8
		\end{align}
		with high probability.

		Now, our problem is reduced to bounding $S$.
		Observe that
		\begin{align*}
			\operatorname{Var}(S) = \sum \limits_{i=2}^{\lfloor q\sqrt{n}/4\rfloor} \dfrac{\operatorname{Var}(J_i)}{i^2} =\Theta(1) \qquad \text{and} \qquad \mathbb{E}(S)=\Theta(\log(n)).
		\end{align*}
		By Chebyshev's inequality\footnote{Chebyshev's inequality states that if $X$ is a random variable and $t >0$, then $\Pr{|X-\Ex{X}|\ge t}\le \operatorname{Var}{X}/t^2$.}, it follows that $S = \Theta(\log(n))$ with high probability, and hence
		it follows from~\eqref{eq:A2intn-final} that
		\begin{align*}
			|A^2\cap [2,n]|  = \Omega(q^2n \log(n)) \gg 1/q.
		\end{align*}
		In the last inequality, we used that $q = \Theta(p)$ and that $p \gg (n \log n)^{-1/3}$. This concludes our proof.
	\end{proof}
	
	\section{Product Schur triples in randomly perturbed sets}
	\label{sec:perturbed}

	For a positive integer $n$ and an interval $I \subseteq [2,n]$, we denote by
	\[ H(n,I) = \big\{x \le n: x = d \cdot y, \text{ for some } d \text{ in } I\big\} \,  \]
	the set of positive integers in $[n]$ that have a divisor in $I$.
	The main tool behind Theorem~\ref{thm:perturbed} is the following result of Ford \cite{Ford2008}. 
	
	\begin{theorem}\label{thm:kevin}
		There is an absolute constant $\gamma \in (0,1)$ such that for any integers $n,y,z$ with $n \ge 10^5$, $100 \le y \le z-1$, $y \le \sqrt{n}$, $2y \le z \le y^2$, and
		\[u= \dfrac{\log (z)}{\log (y)} - 1.\]
		we have
		\begin{align*}
			\gamma nu^{\delta} \left(\log \tfrac{2}{u}\right)^{-3/2} \le |H(n,(y,z))| \le \gamma^{-1} 
   nu^{\delta} \left(\log \tfrac{2}{u}\right)^{-3/2}.
		\end{align*}
	\end{theorem}

    Before we prove Theorem~\ref{thm:perturbed}, we shall need the following claim on the growth of $f$ stated in \eqref{eq:f_alpha} above.

    \begin{claim}\label{eq:growth-of-f} For $\alpha \in (0,2^{-7})$ we have $f(\alpha) \le 1/4$ and
    \begin{align}
    \alpha^{1/\delta} \left( \log\left(2 \alpha^{-1/\delta}\right) \right)^{-3/(2\delta)} \le f(\alpha) \le \alpha^{1/\delta} \left( \log\left(2 \alpha^{-1/\delta}\right) \right)^{3/(2\delta)} \, .
    \end{align}
    \end{claim}

    \begin{proof}
        As the function $g: \mathbb{R}_{> 0} \to \mathbb{R}$ given by $g(z)=z^{-z}$ attains its maximum at $z = e^{-1}$ and as $\delta > 1/20$, for all $z \in \mathbb{R}_{>0}$ we have
        \begin{align*}
            \left ( \dfrac{1}{2z^{3/(2\delta)}} \right )^z = \dfrac{1}{2^z} \cdot \dfrac{1}{z^{3z/(2\delta)}} \le e^{3/(2e\delta)} \le e^{12},
        \end{align*}
        and hence
        \begin{align}\label{eq:tech-1}
            z \log \left ( \dfrac{1}{2z^{3/(2\delta)}} \right ) \le 12
        \end{align}
        for all $z \in \mathbb{R}_{>0}$.
        By setting $y = z^{3/(2\delta)}$, it follows from~\eqref{eq:tech-1} that $y^{2\delta/3} \log ((2y)^{-1}) \le 12$, and hence
        \begin{align}\label{eq:tech-2} 
            y^{\delta} \left ( \log \left ( \dfrac{1}{2y} \right ) \right )^{3/2}\le 12^{3/2} \le 2^6
        \end{align}
        for all $y \in \mathbb{R}_{>0}$.
        In particular, it follows from~\eqref{eq:tech-2} that
        \begin{align}\label{eq:tech-3}
            h(y) \coloneqq y^{\delta} \left ( \log \left ( \dfrac{1}{2y} \right ) \right )^{-3/2} \ge 2^{-6} y^{2\delta}
        \end{align}
        for all $y \in \mathbb{R}_{>0}$.

        By using that $\alpha = 2^{2\delta}h(f(\alpha))$ (by definition of $f$, see~\eqref{eq:defn-f}) and replacing $y=f(\alpha)$ in~\eqref{eq:tech-3},
        we obtain
        \begin{align}\label{eq:tech-4}
            \big (2^{6-2\delta} \alpha \big )^{1/(2\delta)} = \big (2^6 h(f(\alpha)) \big )^{1/(2\delta)} \ge f(\alpha).
        \end{align}
        Thus, it follows from~\eqref{eq:tech-4} that $f(\alpha) \le 1/4$ for all $\alpha \in (0,2^{-6-2\delta})$.

        As $\log (2/\alpha^{1/\delta}) \ge 4$ for all $0 < \alpha < (2e^{-4})^{\delta}$ and we are in the range where $\alpha < 2^{-6-2\delta} < (2e^{-4})^{\delta}$, we obtain
        \begin{align}\label{eq:tech-5}
            \left ( \log \left ( \dfrac{2}{\alpha^{1/\delta}} \right ) \right )^{-3/2} \alpha \le \left ( \log \left ( \dfrac{2}{\alpha^{1/\delta}} \right ) \right )^{-\delta} \alpha \le 2^{-2\delta}\alpha = h(f(\alpha)) \le (f(\alpha))^{\delta}.
        \end{align}
        The lower bound on $f(\alpha)$ follows raising each term in the inequalities above to the power of $1/\delta$.
        For the upper bound, as $2^{-2}\alpha^{1/\delta} \le f(\alpha)$ (by the last inequality in~\eqref{eq:tech-5}), we have
        \begin{align}\label{eq:tech-6}
            \alpha = 4^{\delta} f(\alpha)^{\delta}  \left ( \log \left ( \dfrac{1}{2f(\alpha)} \right ) \right )^{-3/2} \ge 4^{\delta} f(\alpha)^{\delta}  \left ( \log \left ( \dfrac{2}{\alpha^{1/\delta}} \right ) \right )^{-3/2}.
        \end{align}
        The upper bound on $f$ then easily follows from~\eqref{eq:tech-6}.
    \end{proof}
	
	\begin{proof}[Proof of Theorem~\ref{thm:perturbed}]


 
		Let $\gamma>0$ be given by Theorem~\ref{thm:kevin} and let $\alpha$ be such that
        $$(\log(n))^{-\delta} (\log \log(n))^{3/2+\delta}\leq \alpha \leq2^{-7},$$
        Set $y = n^{\frac{1}{2}-\beta(\alpha)}$ and $z = n^{\frac{1}{2}+\beta(\alpha)}$.
        First, let us show that $\gamma \alpha n \le |H(n,(y,z))| \le \gamma^{-1}\alpha n$.
		Note that $2y \le z \le y^2$ if and only if $\sqrt{2} \le n^{\beta(\alpha)} \le n^{1/6}$, which are satisfied by our choice of $\alpha$.
		Moreover, we have $100 \le y \le z-1$ and $y \le \sqrt{n}$, and hence we can apply Theorem~\ref{thm:kevin}.
		Set
		\[u = \dfrac{\log (z)}{\log (y)} -1 = \dfrac{\log(z/y)}{\log (y)} = \dfrac{4\beta(\alpha)}{1-2\beta(\alpha)} = 4f(\alpha).\]
	   The upper and lower bounds on $|H(n,(y,z))|$ then follow from Theorem~\ref{thm:kevin} and the definition of $f(\alpha)$.

         We start by proving item $(i)$.
	   Set
		\begin{align*}
			C_n:= \left[n^{1-2\beta(\alpha)},n\right]\setminus H(n,(y,z)).
		\end{align*}
        Now, we claim that $n^{-2\beta(\alpha)} \le \gamma^{-1} \alpha$.
        In fact, as $f(\alpha)<1$, we have $\beta(\alpha) \ge f(\alpha)/2.$
        Moreover, as $\alpha \ge (\log n)^{-\delta}(\log \log n)^{\frac{3}{2}+\delta}$, it follows from the lower bound on $f(\alpha)$ in Claim~\ref{eq:growth-of-f} that 
        \[ \beta(\alpha) \ge \dfrac{\alpha^{\frac{1}{\delta}}}{2(\log(2\alpha^{-\frac{1}{\delta}}))^{\frac{3}{2\delta}}} \ge \dfrac{(\log (n))^{-1}(\log \log (n))^{\frac{3}{2\delta}+1}}{2(\log(2\alpha^{-\frac{1}{\delta}}))^{\frac{3}{2\delta}}} \ge \dfrac{\log \log (n)}{4\log (n)}.\]
        In the last inequality, we actually only used that $\alpha \ge (\log n)^{-\delta}$ and that $\log (2\log (n)) \le 2 \log \log (n)$. Therefore,
        \[ n^{-2\beta(\alpha)} \le (\log n)^{-1/2} \ll (\log n)^{-\delta}(\log \log n)^{\frac{3}{2}+\delta} \le \alpha.\]
        This proves our claim. 
        Since $n^{1-2\beta(\alpha)} \le \gamma^{-1} \alpha n$, it follows that $|C_n| \ge (1-2\gamma^{-1} \alpha)n$, which is void unless $\alpha < \tfrac 12  \gamma$.


Now, let $2 \le a \le b \le c \le n$ be such that $ab=c$ and suppose that $\{a,b,c\} \cap C_n \not = \emptyset$.
We claim that we must have $a \le n^{\frac{1}{2}-\beta(\alpha)}$.
Indeed, if $\{a,b\} \cap C_n \not = \emptyset$, then this implies that
$b \ge n^{1-2\beta(\alpha)}$, and hence $a \le n^{2\beta(\alpha)}  \le n^{\frac{1}{2}-\beta(\alpha)}$ (as $\beta(\alpha) \le 1/6$ by our choice of $\alpha$).
If $c \in C_n$, then $c \notin H(n,(y,z))$, and hence both $a$ and $b$ do not belong to the interval $(n^{\frac{1}{2}-\beta(\alpha)},n^{\frac{1}{2}+\beta(\alpha)})$.
This implies that $a \le n^{\frac{1}{2}-\beta(\alpha)}$, otherwise we would have $ab > n$.

		As $p$ is much smaller than the threshold for $[2,n]_p$ to contain a product Schur triple, if $C_n \cup [2,n]_p$ contains a product Schur triple, then $[2,n]_p$ contains an element in $[n^{\frac{1}{2}-\beta(\alpha)}]$.
		Thus, if $p \ll n^{-\frac{1}{2}+\beta(\alpha)}$, then
		\[\Pr{C_n \cup [2,n]_p \text{ contains a product Schur triple}} \le \Pr{\big| \big[n^{\frac{1}{2}-\beta(\alpha)}]_p\big|\ge 1}\to 0.\]
		
		For item $(ii)$,
        let $(C_n)_{n \in \mathbb{N}}$ be any sequence such that $|C_n| \ge (1-\tfrac 12 \gamma \alpha)n$.
        By monotonicity, we may assume that $p \ll 1$.
		Then, we have that the set $C'_n := C_n \cap H(n,(y,z))$ has size at least 
        \begin{align*}
            |C'_n| \ge |C_n| + |H(n,(y,z))| - n \ge (1-\tfrac 12 \gamma \alpha)n + \gamma \alpha n - n \ge \tfrac{1}{2} \gamma \alpha n.
        \end{align*}
		Now, let $G$ be a graph with vertex set $[2,n^{\frac{1}{2}+\beta(\alpha)}]$ and edge set $E(G)=\{\{a,b\}: ab \in C'_n\}$. Let $d$ be the average degree of $G$, that is, $d = 2e(G)/v(G)$ and
		set $X = \{v \in V(G): d(v)>d/2\}$.
		Note that
		\[|X|v(G) + v(G)d/2 \ge |X|v(G)+(v(G)-|X|)d/2 \ge d v(G).\]
		This implies that $|X| \ge d/2$. As $e(G)\ge |C'_n|$, it follows that $|X| \ge \gamma \alpha n^{\frac{1}{2}-\beta(\alpha)}/2$.
		
        As containing a product Schur triple is a monotone property, we can assume that $p= f(n) \alpha^{-1} n^{-\frac{1}{2}+\beta(\alpha)}$, where $f(n) \rightarrow \infty$, but $f(n) \le \log (n)$.
        Let $q \in [0,1]$ be such that $(1-q)^2 = 1-p$; note that, as $p \ll 1$, we have that $q$ is asymptotically equal to $p/2$.
        Let $A := [2,n]_{q}$ and $B := [2,n]_{q}$ to be two independent random sets; observe that $A \cup B$ has the same distribution as $[2,n]_p$.
		Now, note that as $p \gg \alpha^{-1} n^{-\frac{1}{2}+\beta(\alpha)}$, we have
		\begin{align}
			\Pr{A \cap X = \emptyset} = e^{-\Omega(|X|p)} = o(1).
		\end{align}
		Therefore, with high probability, we have at least one vertex $v \in A \cap X$.
		Now, in order to have an edge captured by $A \cap B$, it suffices to have $B \cap N(v) \neq \emptyset$.
		As $|N(v)| \ge \gamma \alpha n^{\frac{1}{2}-\beta(\alpha)}/2$ for all $v \in X$, it follows that
		\begin{align}
			\Pr{B \cap N(v) = \emptyset} = e^{-\Omega(|N(v)|p)} = o(1).
		\end{align}
		Therefore, with high probability there exists $\{a,b\} \se [2,n]_p$ such that $ab \in C'_n$. This concludes our proof.
	\end{proof}
	
	\section{Concluding remarks}\label{sec:remarks}

In the deterministic setting, we introduced some new definitions to bridge known results of the sum-free case in our setting.
In particular, following Abbot and Hanson \cite{abbott1972problem} definition of strongly sum-free sets, we re-introduce double-sum Schur numbers $S'(k)$, which we showed to be related to the construction of large product-free sets. However, we did not focus on determining bounds for $S'(k)$, as the problem seems reminiscent of finding bounds for $S(k)$, which proved to be difficult. Still, the following question might be approachable.
	
\begin{problem}
Is there an $\epsilon>0$ such that for $k$ large enough we have $S'(k)<(1-\epsilon) S(k)$?
\end{problem}

For the case $k=2$ we try to determine the minimum number of monochromatic products guaranteed in any $2$-colourings of $[n]$ and show in Theorem \ref{thm:det_prod_mult} a lower bound of $n^{\frac{1}{3}-\epsilon}$. 
In particular, we are interested in the following question originally asked by Prendiville \cite{Prendiville} that partially motivated our work.

\begin{problem}
\label{prob:Prendiville}
For $k$ a positive integer, what is the minimal number of monochromatic products in a $k$-colouring of $[2,n]$ and how does this colouring look?
\end{problem}

In a paper synchronised with ours, Aragão, Chapman, Ortega and Souza~\cite{ACOS23} improve our result to $\left(2^{-3/2}-o(1)\right)n^{1/2} \log (n)$ even in a single colour class, which is tight up to lower order terms.
They also have results for larger $k$, which are tight up to a log factor when $k=3, 4$.

In the probabilistic setting, we analysed the probability threshold of the property of containing a product Schur triple.
However, this question can be extended to multiple colours.
In particular, we propose the following problem, which is already interesting in the case $k=2$.

\begin{problem}
For $k$ a fixed positive integer, what is the threshold in $[2,n]_p$ for the property that any $k$-colouring contains a monochromatic product?
\end{problem}

For any of the problems studied in the sum-free case, we can consider an equivalent question in the product-free setting. We hope that this line of questioning can bring a new perspective to the study of Schur triples and other equations.

 \section*{Acknowledgements}\label{sec:acknowledgements} 
    \noindent This project began when D.M.C.~visited the other two authors at Freie Universität Berlin.
    We express our gratitude to Freie Universität Berlin for their warm hospitality and the London Mathematical Society for its financial support that allowed the visit to take place.
    We thank Sean Prendiville for presenting the problem of counting monochromatic product Schur triples in the integers at the British Combinatorial Conference 2022.
    Finally, we thank both referees for their careful reading and many comments which greatly improved the presentation of this paper.

 \bibliographystyle{abbrv}
	
\small{\bibliography{prod-free}}

\begin{thebibliography}{10}

\bibitem{abbott1972problem}
H.~Abbott and D.~Hanson.
\newblock A problem of {S}chur and its generalizations.
\newblock {\em Acta Arith.}, 2(20):175--187, 1972.

\bibitem{AbbottWang1977}
H.~L. Abbott and E.~T.~H. Wang.
\newblock Sum-free sets of integers.
\newblock {\em Proc. of the Amer. Math. Soc.}, 67(1):11--16, 1977.

\bibitem{aigner2019monochromatic}
E.~Aigner-Horev and Y.~Person.
\newblock Monochromatic {S}chur triples in randomly perturbed dense sets of integers.
\newblock {\em SIAM J. on Discr. Math.}, 33(4):2175--2180, 2019.

\bibitem{ACOS23}
L.~Aragão, J.~Chapman, M.~Ortega, and V.~Souza.
\newblock On the number of monochromatic solutions to multiplicative equations.
\newblock {\em preprint}, 2023.

\bibitem{bohman2003many}
T.~Bohman, A.~Frieze, and R.~Martin.
\newblock How many random edges make a dense graph {H}amiltonian?
\newblock {\em Random Struct. Algorithms}, 22(1):33--42, 2003.

\bibitem{bollobas1987threshold}
B.~Bollob{\'a}s and A.~G. Thomason.
\newblock Threshold functions.
\newblock {\em Combinatorica}, 7(1):35--38, 1987.

\bibitem{CameronErdos}
P.~Cameron and P.~Erd\H{o}s.
\newblock {\em On the Number of Sets of Integers With Various Properties}, pages 61--80.
\newblock De Gruyter, Berlin, Boston, 1990.

\bibitem{das2023schur}
S.~Das, C.~Knierim, and P.~Morris.
\newblock Schur properties of randomly perturbed sets.
\newblock {\em European J. Combin.}, pages 1--25, 2023.

\bibitem{Ford2008}
K.~Ford.
\newblock The distribution of integers with a divisor in a given interval.
\newblock {\em Ann. Math.}, 168(1):367--433, 2008.

\bibitem{GrahamRodlRucinski1996}
R.~Graham, V.~R{\"o}dl, and A.~Ruci{\'n}ski.
\newblock On {S}chur properties of random subsets of integers.
\newblock {\em J. Number Theory}, 61(2):388--408, 1996.

\bibitem{heule2018schur}
M.~Heule.
\newblock {S}chur number five.
\newblock In {\em Proceedings of the AAAI Conference on Artificial Intelligence}, volume~32, 2018.

\bibitem{Hu1980}
M.~C. Hu.
\newblock A note on sum-free sets of integers.
\newblock {\em Proc. of the Amer. Math. Soc.}, 80(4):711--712, 1980.

\bibitem{Irving1973}
R.~Irving.
\newblock An extension of {S}chur's theorem on sum-free partitions.
\newblock {\em Acta Arith.}, 25(1):55--64, 1973.

\bibitem{Prendiville}
S.~Prendiville.
\newblock Problem sessions of the 29th British Combinatorial Conference in Lancaster, 2022.

\bibitem{Robertson1998}
A.~Robertson and D.~Zeilberger.
\newblock A 2-coloring of $[1, n]$ can have $n^2/22+ {O}(n)$ monochromatic {S}chur triples, but not less!
\newblock {\em Electron. J. Combin.}, 5:1--4, 1998.

\bibitem{schoen1999number}
T.~Schoen.
\newblock The number of monochromatic {S}chur triples.
\newblock {\em European J. Combin.}, 20(8):855--866, 1999.

\bibitem{Schur1917}
I.~Schur.
\newblock {\"U}ber kongruenz x...(mod. p.).
\newblock {\em Jahresber. Dtsch. Math.-Ver.}, 25:114--116, 1917.

\bibitem{wigert1907ordre}
S.~Wigert.
\newblock Sur l'ordre de grandeur du nombre des diviseurs d'un entier.
\newblock {\em Arkiv f\"{o}r Mathematik}, 3(18):1--9, 1906-1907.

\end{thebibliography}

\end{document}